\documentclass[12pt]{article}

\usepackage{enumerate}
\usepackage{amsthm,amsmath,amssymb}

\usepackage{graphicx}

\usepackage[colorlinks=true,citecolor=black,linkcolor=black,urlcolor=blue]{hyperref}

\def\eps{\varepsilon}

\theoremstyle{plain}
\newtheorem{theorem}{Theorem}
\newtheorem{lemma}[theorem]{Lemma}
\newtheorem{corollary}[theorem]{Corollary}
\newtheorem{proposition}[theorem]{Proposition}

\theoremstyle{definition}
\newtheorem{definition}[theorem]{Definition}

\theoremstyle{remark}

\title{\bf Subgraphs with large minimum $\ell$-degree in hypergraphs where almost all $\ell$-degrees are large}

\author{Victor Falgas-Ravry\thanks{Supported by VR  starting grant 2016-03488} \\
	\small Institutionen f\"or matematik och matematisk statistik\\[-0.8ex]
	\small Ume{\aa} Universitet\\[-0.8ex]
	\small  Ume{\aa}, Sweden\\
	\small \tt victor.falgas-ravry@umu.se\\
	\and
	Allan Lo\thanks{Supported by EPSRC first grant EP/P002420/1} \\
	\small Department of Mathematics\\ [-0.8ex]
	\small University of Birmingham\\ [-0.8ex]
	\small Birmingham, England\\
	\small \tt s.a.lo@bham.ac.uk}

\begin{document}
	
	\maketitle
	
	
	\begin{abstract}
	
	Let $G$ be an $r$-uniform hypergraph on $n$ vertices such that all but at most $\eps \binom{n}{\ell}$ $\ell$-subsets of vertices have degree at least $p \binom{n-\ell}{r-\ell}$.
	We show that $G$ contains a large subgraph with high minimum $\ell$-degree.
			
	\end{abstract}
	\section{Introduction}\label{secion: introduction}
	Given $r \in \mathbb{N}$ and a set~$A$, we write $A^{(r)}$ for the collection of all $r$-subsets of~$A$ and $[n]$ for the set~$\{1,2,\ldots n\}$. An \emph{$r$-graph}, or \emph{$r$-uniform hypergraph}, is a pair $G=(V,E)$, where $V=V(G)$ is a set of vertices and $E=E(G)\subseteq V^{(r)}$ is a collection of $r$-subsets, which constitute the edges of~$G$. We say $G$ is \emph{nonempty} if it contains at least one edge and set $v(G)=\vert V(G)\vert$ and $e(G)=\vert E(G)\vert$. A \emph{subgraph} of~$G$ is an $r$-graph~$H$ with $V(H)\subseteq V(G)$ and $E(H)\subseteq E(G)$. The subgraph of $G$ \emph{induced} by a set $X\subseteq V(G)$ is $G[X]=(X, E(G)\cap X^{(r)})$. 

Let $\mathcal{F}$ be a family of nonempty $r$-graphs. 
If $G$ does not contain a copy of a member of~$\mathcal{F}$ as a subgraph, we say that $G$ is \emph{$\mathcal{F}$-free}. 
The \emph{Tur\'an number} $\mathrm{ex}(n, \mathcal{F})$ of a family~$\mathcal{F}$ is the maximum number of edges in an $\mathcal{F}$-free $r$-graph on $n$ vertices, and its \emph{Tur\'an density} is the limit $\pi(\mathcal{F})=\lim_{n\rightarrow \infty} \mathrm{ex}(n,\mathcal{F})/\binom{n}{r}$ (this is easily shown to exist). Let $K_t^{(r)}=([t], [t]^{(r)})$ denote the complete $r$-graph on $t$ vertices. Determining $\pi(K_t^{(r)})$ for \emph{any} $t>r\ge3$ is a major problem in extremal combinatorics. Tur\'an~\cite{MR0018405} famously conjectured in 1941 that $\pi(K_4^{(3)})= 5/9$, and despite much research effort this remains open~\cite{MR2866732}.
In this paper we shall be interested in some variants of Tur\'an density.

The \emph{neighbourhood $N(S)$} of an $\ell$-subset $S \in V(G)^{(\ell)}$ is the collection of $(r-\ell)$-subsets $T \in V(G)^{(r-\ell)}$ such that $S \cup T$ is an edge of~$G$. The \emph{degree} of~$S$ is the number $\deg(S)$ of edges of~$G$ containing~$S$, that is, $\deg(S) = |N(S)|$.
 The minimum $\ell$-degree of~$G$, $\delta_{\ell}(G)$, is defined to be the minimum of $\deg(S)$ over all $\ell$-subsets $S\in V(G)^{(\ell)}$. 
The \emph{Tur\'an $\ell$-degree threshold} $\mathrm{ex}_{\ell}(n,\mathcal{F})$ of a family~$\mathcal{F}$ of $r$-graphs is the maximum of~$\delta_{\ell}(G)$ over all $\mathcal{F}$-free $r$-graphs~$G$ on $n$ vertices. It can be shown~\cite{MubayiZhao07, LoMarkstrom12} that the limit $\pi_\ell(\mathcal{F})=\lim_{n\rightarrow \infty} \mathrm{ex}_\ell(n,\mathcal{F})/\binom{n-\ell}{r-\ell}$ exists; this quantity is known as the \emph{Tur\'an $\ell$-degree density} of~$\mathcal{F}$. A simple averaging argument shows that 
	\[0 \leq \pi_{r-1}(\mathcal{F}) \le \dots \le \pi_2(\mathcal{F}) \leq \pi_1 (\mathcal{F}) = \pi(\mathcal{F}) \leq 1,\]
	and it is known that $\pi_\ell(\mathcal{F})\neq \pi(\mathcal{F})$ in general (for $\ell \notin \{0,1\}$). In the special case where $(r, \ell)=(r, r-1)$, $\pi_{r-1}(\mathcal{F})$ is known as the \emph{codegree density} of $\mathcal{F}$.

There has been much research on Tur\'an $\ell$-degree threshold for $r$-graphs when $(r,\ell) = (3,2)$.  
In the late 1990s, Nagle~\cite{Nagle99} and Nagle and Czygrinow~\cite{CzygrinowNagle01} conjectured that $\pi_2(K_4^{(3)-})=1/4$ and $\pi_2(K_4^{(3)})=1/2$, respectively. Here $K_4^{(3)-}$ denotes the $3$-graph obtained by removing one edge from $K_4^{(3)}$. Falgas-Ravry, Pikhurko, Vaughan and Volec~\cite{FalgasRavryPikhurkoVaughanVolec17, FalgasRavryPikhurkoVaughanVolec16+} recently proved $\pi_{2}(K_4^{(3)-})=1/4$, settling the conjecture of Nagle, and showed all near-extremal constructions are close (in edit distance) to a set of quasirandom tournament constructions of Erd{\H o}s and Hajnal~\cite{ErdosHajnal72}. The lower bound $\pi_2(K_4^{(3)})\geq 1/2$ also comes from a quasirandom construction, which is due to R\"odl~\cite{Rodl86}. For $t>r\geq 3$, the codegree density $\pi_{r-1}(K_t^{(r)})$ has been studied by Falgas-Ravry~\cite{FalgasRavry13}, Lo and~Markstr\"om~\cite{LoMarkstrom12} and Sidorenko~\cite{Sid17}. Recently, Lo and Zhao~\cite{LoZhao18} showed that $1- \pi_{r-1}(K_t^{(r)}) = \Theta (\ln t /t^{r-1}) $ for $r \ge 3$.


One variant of $\ell$-degree Tur\'an density is to study $r$-graphs in which almost all $\ell$-subsets have large degree. 
To be precise, given $\eps >0$, let $\delta_{\ell}^{\eps}(G)$  be the largest integer $d$ such that all but at most $\eps \binom{v(G)}{\ell}$ of the $\ell$-subsets $S \in V(G)^{(\ell)}$ satisfy $\deg(S) \ge d$. 
Note that $r$-graphs with large $\delta_{\ell}^{\eps}(G)$ but with small $\delta_{\ell} (G)$ arise naturally. 
For instance, the reduced graphs~$R$ obtained from $r$-graphs with large minimum $\ell$-degree after an application of hypergraph regularity lemma have large $\delta^{\eps}_{\ell}(R)$.

\begin{definition}[$(r,\ell)$-sequence] \label{def:1}
		Let $1 \le \ell < r$.
		We say that a sequence $\mathbf{G}=(G_n)_{n\in \mathbb{N}}$ of $r$-graphs is an \emph{$(r,\ell)$-sequence} if 
		\begin{enumerate}[(i)]
			\item $v(G_n)\rightarrow \infty$ as $n\rightarrow \infty$ and
			\item there is a constant $p \in [0,1]$ and a sequence of nonnegative reals $\varepsilon_n \rightarrow 0$ as $n \rightarrow \infty$ such that $\delta_{\ell}^{\eps_n}(G_n) \ge p \binom{v(G_n)-\ell}{r-\ell}$ for each $n$. 
	\end{enumerate}
		We refer to the supremum of all $p\geq 0$ for which (ii) is satisfied as the \emph{density} of the sequence $\mathbf{G}$ and denote it by $\rho(\mathbf{G})$ .
	\end{definition}	
	
	We can define the analogue of Tur\'an density for $(r,\ell)$-sequences.
	\begin{definition} \label{def:2}
		Let $1 \le \ell < r$.		
		Let $\mathcal{F}$ be a family of nonempty $r$-graphs. 
		Define
		\[\pi^{\star}_{\ell}(\mathcal{F}):=\sup \Bigl\{ \rho(\mathbf{G}): \ \mathbf{G} \textrm{ is an $(r,\ell)$-sequence of $\mathcal{F}$-free $r$-graphs} \Bigr\}.\]
	\end{definition}

Our main result show that every large $r$-graph~$G$ contains a `somewhat large' subgraph~$H$ with minimum $\ell$-degree satisfying $\delta_{\ell}(H) / \binom{v(H)-\ell}{r-\ell} \approx \delta_{\ell}^{\eps}(G) / \binom{v(G)-\ell}{r-\ell}$. Here `somewhat large' means $v(H)= \Omega(\eps^{1/\ell})$.
		
\begin{theorem}\label{theorem: high codeg extraction from wqr 3-graphs}
Let $1 \le \ell < r$. 
For any fixed $\delta>0$, there exists $m_0>0$ such that any $r$-graph~$G$ on $n \geq m \ge m_0$ vertices with $\delta_{\ell}^{\eps} (G) \ge p \binom{n-\ell}{r- \ell}$ for some $\eps\leq m^{-\ell}/2$ contains an induced subgraph~$H$ on $m$ vertices with 
\begin{align*}
	\delta_{\ell} (H) \geq (p-\delta) \binom{ m - \ell }{r - \ell }.
\end{align*}
\end{theorem}

This immediate implies the $\pi_{\ell}^{\star}(\mathcal{F}) = \pi_{\ell}(\mathcal{F})$ for all families $\mathcal{F}$ of $r$-graphs.
\begin{corollary}\label{corollary: pi equals pi star}
	For any $1 \le \ell < r$ and any family~$\mathcal{F}$ of nonempty $r$-graphs, $\pi^{\star}_{\ell}(\mathcal{F}) = \pi_{\ell}(\mathcal{F})$.
\end{corollary}
We note that the (tight) upper bounds for codegree densities $\pi_2(F)$ for $3$-graphs $F$ obtained by flag algebraic methods in~\cite{FalgasRavryMarchantPikhurkoVaughan15, FalgasRavryPikhurkoVaughanVolec17, FalgasRavryPikhurkoVaughanVolec16+} actually relied on giving upper bounds for $\pi^{\star}_{\ell}(F)$. Corollary~\ref{corollary: pi equals pi star} provides theoretical justification for why this strategy could give optimal bounds.

\subsection{Quasirandomness in $3$-graphs}
One of the main motivations for this note comes from recent work of Reiher, R\"odl and Schacht~\cite{ReiherRodlSchacht16a, ReiherRodlSchacht16d,  ReiherRodlSchacht16b, ReiherRodlSchacht16c} on extremal questions for quasirandom hypergraphs.
These authors studied the following notion of quasirandomness for $3$-graphs. 
	
\begin{definition}[(1,2)-quasirandomness]
A $3$-graph $G$ is \emph{$(p, \varepsilon, (1,2))$-quasirandom} if for every set of vertices $X\subseteq V$ and every set of pairs of vertices $P\subseteq V^{(2)}$, the number $e_{1,2}(X,P)$ of pairs $(x, uv)\in X\times P$ such  that $\{x\}\cup \{uv\}\in E(G)$ satisfies:
		\[\Bigl\vert e_{1,2}(X, P) - p \vert X\vert \cdot \vert P\vert \Bigr\vert \leq \varepsilon {v(G)}^3.\]
	\end{definition}

We define a $(1,2)$-quasirandom sequence and the corresponding extremal density, denoted by $\pi_{(1,2)-qr}(\mathcal{F})$, analogously to the way we defined $(r,\ell)$-sequences and $\pi_{\ell}^{\star}(\mathcal{F})$ in Definitions~\ref{def:1} and~\ref{def:2}.
It is not difficult to see that $\pi_{(1,2)-qr}(\mathcal{F}) \le \pi (\mathcal{F})$ for all families $\mathcal{F}$ of $3$-graphs. 
Moreover, a $(p, \varepsilon, (1,2))$-quasirandom $3$-graph~$G$ satisfies $\delta_{2}^{ \sqrt{\eps}}(G) \ge (p - 4\sqrt{\eps}) v(G)$. 
Hence, Theorem~\ref{theorem: high codeg extraction from wqr 3-graphs} and Corollary~\ref{corollary: codegree density is an ub for wqr density} imply the following. 
\begin{corollary}\label{corollary: codegree density is an ub for wqr density}
For any family of nonempty $3$-graphs $\mathcal{F}$, $\pi_{(1,2)-qr}(\mathcal{F})\leq \pi_2(\mathcal{F})$. 
\end{corollary}
Consider a $(p, \varepsilon, (1,2))$-quasirandom $3$-graph~$G$  for some $p>4\sqrt{\eps}>0$.  As noted above,  $\delta_{2}^{ \sqrt{\eps}}(G) \ge (p- 4\sqrt{\eps}) v(G)$. Thus provided $v(G)$ is sufficiently large, Theorem~\ref{theorem: high codeg extraction from wqr 3-graphs} tells us we can find a subgraph $H$ of $G$ on $m=\Omega(\eps^{-1/4})$ vertices with strictly positive minimum codegree (at least $(p-4\sqrt{\eps})m$).

However, as we show below, we cannot guarantee the existence of any subgraph with strictly positive codegree on more than $2/\varepsilon+1$ vertices: our lower bound on~$m$ above in terms of an inverse power of the error parameter $\varepsilon$ is thus sharp up to the value of the exponent. 
	\begin{proposition}\label{proposition: wqr graph with no large subgraphs with codeg density>0}
		For every $p\in(0,1)$ and every $\varepsilon>0$, there exists $n_0$ such that for all $n\geq n_0$ there exist $(p, 2\varepsilon, (1,2))$-quasirandom $3$-graphs in which every subgraph on  $m\geq \lfloor \varepsilon^{-1} \rfloor +1$ vertices has minimum codegree equal to zero. 
	\end{proposition}
	
	\begin{proof}
		Let $G=(V,E)$ be a $(p, \varepsilon, (1,2))$-quasirandom $3$-graph on $n$ vertices. Such a $3$-graph can be obtained for example by taking a typical instance of an Erd{\H o}s--R\'enyi random $3$-graph with edge probability $p$.
		Consider a balanced partition of $V$ into $N=\lfloor \varepsilon^{-1}\rfloor$ sets $V=\bigcup_{i=1}^N V_i$ with $\lfloor n/N\rfloor \leq \vert V_1\vert\leq \vert V_2\vert \leq \ldots \leq \vert V_N\vert \leq \lceil n/N\rceil$. Now let $G'$ be the $3$-graph obtained from $G$ by deleting all triples that meet some $V_i$ in at least two vertices for some $i$: $1\leq i \leq N$. 
		
		By construction, every set of $N+1$ vertices in $G'$ must contain at least two vertices from the same $V_i$, and thus must induce a subgraph of $G'$ with minimum codegree zero.
		Note that $e(G) - e(G') \le N n \binom{\lceil n/N\rceil}{2} \le n^3/N \le \varepsilon n^3$.
		Since $G$ is $(p, \varepsilon, (1,2))$-quasirandom, it follows that $G'$ is $(p, 2\varepsilon, (1,2))$-quasirandom.
	\end{proof}

	\section{Finding high minimum $\ell$-degree subgraphs in $r$-graphs with large~$\delta_{\ell}^{\varepsilon}$}

	In this section we show how we can extract arbitrarily large subgraphs with high minimum $\ell$-degree from sufficiently large $r$-graphs with sufficiently small error $\varepsilon$.
	To do so, we will need Azuma's inequality (see e.g.~\cite{MR1885388}).

\begin{lemma}[Azuma's inequality] \label{lem:azuma}
Let $\{ X_i: i=0, 1, \dots\}$ be a martingale with $|X_i - X_{i-1}|  \le c_i$ for all $i$.
Then for all positive integers $N$ and $\lambda >0$,
\begin{align*}
\mathbb{P}(X_N \le X_0 - \lambda) \le \exp \left( \frac{-\lambda^2}{2 \sum_{i=1}^N c_i^2}\right).
\end{align*}
\end{lemma}

\begin{proof}[Proof of Theorem~\ref{theorem: high codeg extraction from wqr 3-graphs}]
We may assume without loss of generality that  $\delta>0$ is small enough to ensure $\delta^{-1} \ge 26 \ell (r - \ell)^2 \log (1/\delta)$ and $ \ell \log (1/\delta) \ge \log 2$ as this only makes our task harder.
Set $m_0 = \left\lceil  26 \ell (r - \ell)^2 \delta^{-2} \log (1/\delta) \right\rceil$.
Note that this implies that 
\begin{align}
\label{eqn:m_0}
		2\ell \log m_0 & \le 4 \ell \log \left(  26 \ell (r - \ell)^2 \delta^{-2} \log (1/\delta) \right)
		\le 12 \ell \log (1/\delta).
\end{align}
Fix $m \geq m_0$. 
Let $n\geq m\geq m_0$ and $\varepsilon = m^{-\ell}/2$.

Suppose $G=(V,E)$ is an $r$-graph on $n$ vertices with $\delta^{\eps}_{\ell}(G) \ge  p \binom{n - \ell}{r - \ell}$.
We claim that it contains an induced subgraph on $m$ vertices with minimum $\ell$-degree at least $(p-\delta)\binom{m - \ell}{r - \ell}$.
For $p\leq \delta$, we have nothing to prove, so we may assume that $1\geq p>\delta$.

Call an $\ell$-subset $S \in V^{(\ell)}$ \emph{poor} if $\deg (S) <  p \binom{n - \ell}{r - \ell}$, and \emph{rich} otherwise.
Let $\mathcal{P}$ be the collection of all poor $\ell$-subsets.
By our assumption on $\delta^{\eps} _{\ell}(G)$, $\vert \mathcal{P}\vert \leq \varepsilon \binom{n}{\ell}$.
As each poor $\ell$-subset is contained in $\binom{n-\ell}{m-\ell}$ $m$-subsets, it follows that there are at least
	\begin{align}\label{eq: lower bound on rich m-subsets}
		\binom{n}{m}- \vert \mathcal{P} \vert \binom{n-\ell}{m-\ell}
		> \bigl(1-  \varepsilon m^\ell \bigr)\binom{n}{m} 
		= \frac12\binom{n}{m}
		\end{align}
$m$-subsets of vertices which do not contain any poor $\ell$-subsets.
		
Given an $\ell$-subset $ S \in V^{(\ell)} \setminus \mathcal{P} $, we call an $m$-subset~$T$ of~$V$ \textit{bad for}~$S$ if $S \subseteq T$ and $\left\vert N(S) \cap T^{(r-\ell)} \right\vert \le (p-\delta)\binom{m- \ell}{r - \ell}$.
Let $\phi_S$ be the number of bad $m$-subsets for~$S$.
We claim that 
\begin{align}
	\phi_S
	\le \binom{ n - \ell }{ m - \ell } \exp \left( - \frac{\delta^2 m }{2(r-\ell)^2} \right). \label{eqn:eqnS}
\end{align}
Observe that
\begin{align*}
	\phi_S = \left\vert \left\{ 
 T \in  (V \setminus S )^{(m-\ell)} \colon \left| N (S) \cap T^{(r-\ell)} \right| \le (p-  \delta) \binom{m- \ell}{r - \ell}
		\right\} \right\vert.
\end{align*}
Let $X$ be the random variable $\left| N (S) \cap T^{(r-\ell)} \right|$, where $T$ is an $(m-\ell)$-subset of~$V \setminus S$ picked uniformly at random.
We consider the vertex exposure martingale on $T$.
Let $Z_i$ be the $i$th exposed vertex in $T$.
Define $X_i = \mathbb{E}(X|Z_1, \dots, Z_i)$.
Note that $\{ X_i: i=0, 1, \dots, m - \ell \}$ is a martingale and $X_0 \ge p \binom{ m - \ell }{r-\ell}$.
Moreover, $|X_i - X_{i-1}| \le  \binom{m-\ell-1}{r-\ell-1} < \binom{m-1}{r - \ell - 1}$.
Thus, by Lemma~\ref{lem:azuma} applied with $\lambda = \delta \binom{m}{r-\ell}$ and $c_i = \binom{m-1}{r - \ell - 1} $, we have
\begin{align*}
	\mathbb{P} \left( X_m \le (p- \delta) \binom{ m - \ell }{r-\ell}\right) 
	& \le \mathbb{P}(X_m \le X_0 - \lambda)
	\le  \exp \left( \frac{-\delta^2 \binom{m}{r-\ell}^2}{2 m \binom{m-1}{r-\ell-1}^2}\right) 
	= \left( \frac{-\delta^2 \binom{m}{r-\ell}}{2 (r-\ell)}\right) \\
	& \le \exp \left( - \frac{ \delta^2 m}{2(r-\ell)^2} \right).
\end{align*}
Hence \eqref{eqn:eqnS} holds.

An $m$-subset~$T$ of~$V$ is called \textit{bad} if it is bad for some~$S \in V^{(\ell)} \setminus \mathcal{P} $.
The number of bad $m$-subsets is at most 
\begin{align*}
	  \sum_{ S \in V^{(\ell)} \setminus \mathcal{P} }  \phi_S 
		& \le \binom{n}{\ell}   \binom{n- \ell}{m - \ell} \exp \left( - \frac{\delta^2 m }{2(r-\ell)^2} \right)
		= \binom{n}{m} \binom{m}{\ell } \exp \left( - \frac{\delta^2 m }{2(r-\ell)^2} \right) \\
		& \le  \binom{n}{m}  m_0^{\ell}  \exp \left( - \frac{\delta^2 m_0 }{2(r-\ell)^2} \right)
		\le \binom{n}m \exp \left( 2\ell \log m_0 - 13 \ell \log (1/\delta) \right)\\
		&  \le \binom{n}m \exp \left( - \ell \log (1/\delta) \right) \le \frac{1}{2} \binom{n}{m},
		\end{align*}
where the last three inequalities hold by our choice of $m_0$, by inequality~\eqref{eqn:m_0}, and by our assumption on $\delta$, respectively.
Together with (\ref{eq: lower bound on rich m-subsets}), this shows there exists an $m$-subset inside which there is no poor $\ell$-subsets and in which every rich $\ell$-subset has degree at least $(p-\delta ) \binom{m- \ell }{r - \ell}$.
Such a set clearly gives us an induced subgraph of $G$ on $m$ vertices with minimum $\ell$-degree at least $(p-\delta)\binom{m- \ell }{r - \ell}$.
	\end{proof}

\section{Concluding remarks}

A $3$-graph $G$ is \emph{$(p, \varepsilon, (1,1,1))$-quasirandom} if for every triple of sets of vertices $X$, $Y$ and $Z\subseteq V$, the number $e_{1,1,1}(X,Y,Z)$ of triples $(x, y,z)\in X\times Y\times Z$ such  that $xyz\in E(G)$ satisfies $\Bigl\vert e_{1,1,1}(X, Y, Z) - p \vert X\vert \cdot \vert Y\vert \cdot \vert Z\vert \Bigr\vert \leq \varepsilon {v(G)}^3$. Define $\pi_{(1,1,1)-qr}(\mathcal{F})$ analogously to $\pi_{(1,2)-qr}(\mathcal{F})$.
Note that $\pi_{(1,2)-qr}(\mathcal{F}) \le \pi_{(1,1,1)-qr}(\mathcal{F}) \le \pi(\mathcal{F})$ for all $3$-graph families $\mathcal{F}$.
An obvious open question is whether we have
	\[\pi_{(1,1,1)-qr}(\mathcal{F})\leq \pi_2(\mathcal{F}).\]
Even more: can one always extract subgraphs with large minimum codegree from $(1,1,1)$-quasirandom graphs? Even obtaining large subgraphs with non-zero minimum codegree remains an open problem for this weaker notion of quasirandomness. 
	
\section*{Acknowledgements}
	The authors are grateful for a Scheme 4 grant from the London Mathematical Society which allowed Victor Falgas-Ravry to visit Allan Lo in Birmingham in July 2016, when this research was done.

	Further, the authors would like to thank anonymous referees at the Electronic Journal of Combinatorics for their careful work and helpful suggestions,  which led to considerable improvements in the paper ---  in particular, their comments led us to state and prove a much more general form of Theorem~\ref{theorem: high codeg extraction from wqr 3-graphs} than we had in the first version of this paper.


\end{document}